\documentclass[12pt]{amsart}

\usepackage[usenames]{color}
\usepackage{amsmath,graphicx,fullpage,enumerate}
\usepackage[colorlinks=true,
linkcolor=webgreen,
filecolor=webbrown,
citecolor=webgreen]{hyperref}
\usepackage{float}
\usepackage{graphicx}

\definecolor{webgreen}{rgb}{0,.5,0}
\definecolor{webbrown}{rgb}{.6,0,0}

\def\la{\langle}
\def\ra{\rangle}
\DeclareMathOperator{\het}{ht}

\usepackage{url}

\newcommand{\seqnum}[1]{\href{http://oeis.org/#1}{\underline{#1}}}

\theoremstyle{plain}
\newtheorem{theorem}{Theorem}
\newtheorem{corollary}{Corollary}
\newtheorem{lemma}{Lemma}
\newtheorem{proposition}{Proposition}

\theoremstyle{definition}

\newtheorem{problem}{Problem}

\theoremstyle{remark}
\newtheorem{remark}{Remark}

\makeatletter
\def\moverlay{\mathpalette\mov@rlay}
\def\mov@rlay#1#2{\leavevmode\vtop{%
   \baselineskip\z@skip \lineskiplimit-\maxdimen
   \ialign{\hfil$\m@th#1##$\hfil\cr#2\crcr}}}
\newcommand{\charfusion}[3][\mathord]{
    #1{\ifx#1\mathop\vphantom{#2}\fi
        \mathpalette\mov@rlay{#2\cr#3}
      }
    \ifx#1\mathop\expandafter\displaylimits\fi}
\makeatother

\newcommand{\dcup}{\charfusion[\mathbin]{\cup}{\cdot}}
\newcommand{\dbigcup}{\charfusion[\mathop]{\bigcup}{\cdot}}

\newcommand{\lam}{\lambda}
\newcommand{\vep}{\varepsilon}

\newcommand{\sT}{{\mathcal T}}
\newcommand{\sU}{{\mathcal U}}
\newcommand{\sA}{\mathcal{A}}

\newcommand{\sD}{\mathcal{D}}
\newcommand{\sG}{\mathcal{G}}

\newcommand{\ld}{\ldots}

\newcommand{\mymod}{{\mbox{mod}\;}}
\newcommand{\bbZ}{\mathbb{Z}}
\newcommand{\bbN}{\mathbb{N}}
\DeclareMathOperator{\lcm}{lcm}

\begin{document}

\title{Natural exact covering systems and the reversion \\of the M\"obius series}

\date{\today}

\author[1]{I. P. Goulden}
\address[1]{Department of Combinatorics and Optimization, University of Waterloo,
Waterloo, ON  N2L 3G1 Canada}
\email{ipgoulde@uwaterloo.ca}

\author[2]{Andrew Granville}
\address[2]{D\'epartement de math\'ematiques et de statistiques,
Universit\'e de Montr\'eal, CP 6128 Succursale Centre-Ville,
Montr\'eal, QC  H3C 3J7 Canada}
\email{andrew@dms.umontreal.ca}

\author[3]{L. Bruce Richmond} 
\address[3]{Department of Combinatorics and Optimization, University of Waterloo,
Waterloo, ON  N2L 3G1 Canada}
\email{lbrichmond@uwaterloo.ca}

\author[4]{Jeffrey Shallit}
\address[4]{School of Computer Science, University of Waterloo,
Waterloo, ON  N2L 3G1 Canada}
\email{shallit@uwaterloo.ca}

\thanks{The work of IPG and JS was supported by NSERC Discovery Grants.}



\begin{abstract}
We prove that the number of natural exact covering systems of cardinality $k$ is equal to the coefficient of $x^k$ in the reversion of the power series~$\sum_{k \ge 1} \mu (k) x^k$, where~$\mu(k)$ is the usual number-theoretic M\"obius function.  Using this result, we deduce an asymptotic expression for the number of such systems.
\end{abstract}

\maketitle

\section{Introduction}
\label{sect1}

A {\it covering system} or {\it complete residue system} is a
collection of finitely many residue classes such that every integer
belongs to at least one of the classes.  The concept was introduced
by Erd\H{o}s \cite{e50,e52}, who used the system of $6$
congruences
\begin{align*}
x &\equiv 0 \pmod 2 &
x & \equiv 0 \pmod 3 \\
x & \equiv 1 \pmod 4 &
x & \equiv 3 \pmod 8 \\
x & \equiv 7 \pmod {12} &
x & \equiv 23 \pmod {24}
\end{align*}
to prove that there exists an infinite arithmetic progression of
odd integers, each of which is not representable
as the sum of a prime and a power of two.  Since 1950, hundreds of papers have been written on covering systems; for surveys of the topic, see, for example \cite{p81,z82,ps02,s05}.

A covering system is called {\em exact} if every integer belongs to one and only one of the
given congruences.  (The system of Erd\H{o}s above is {\em not\/} exact, because the integer $19$ belongs to both residue classes
$3 \pmod {8}$ and $7 \pmod {12}$.)  Exact covering systems, or ECS, are sometimes also called {\em exactly covering systems} or {\em disjoint covering
systems} in the literature (e.g., \cite{n71,f72,f73,nz74,z75,s86}).

Among the exact covering systems, one particular subclass that has received attention consists of the so-called {\em natural exact covering systems}, or NECS.  These are the exact covering systems that can be obtained, starting from the single congruence $x \equiv 0$ (mod $1$), by a finite  number of applications of the following transformation: for some $r\ge 2$,  remove a single congruence $x \equiv a$ (mod $n$) from the system and replace it with the $r$ new congruences $x \equiv a+jn$ (mod $rn$), $j = 0,1, \ldots, r-1$; the value of $r$ can vary for the different applications of this transformation.

The first paper that we could find describing natural exact covering systems is Porubsk{\'y}~\cite{p74}, but he credits Zn{\'a}m with introducing them in an earlier unpublished manuscript. Since then, they have been studied by a number of others, e.g., Burshtein~\cite{b76a,b76b}, Zn{\'a}m~\cite{z82} and Korec~\cite{k84}, but up to now it appears that nobody has enumerated them. Let $a_k$ denote the number of NECS consisting of~$k$ congruences, $k\ge 1$. Define the formal power series
\begin{equation}\label{fpsAM}
A(x) = \sum_{k\ge 1} a_k x^k,\qquad \qquad M(x) = \sum_{k\ge 1} \mu(k) x^k,
\end{equation}
where $\mu$ is the usual number-theoretic {\em M\"obius function} defined by~$\mu(1)=1$ and, for~$n\ge 2$,
$$ \mu(n) = \begin{cases}
0, & \text{ if $n$ is divisible by a square $> 1$;} \\
(-1)^e, & \text{ if $n$ is the product of $e$ distinct primes}.
\end{cases}
$$
The initial terms of these series are given by
\begin{align*}
A(x) &=  x+x^2+3x^3+10x^4+39x^5+160x^6+ 691 x^7 + 3081 x^8 + \ld ,   \\
M(x) &= x-x^2-x^3-x^5+x^6-x^7+x^{10}-x^{11}-x^{13}+x^{15}+\ld,
\end{align*}
where the coefficients in the generating series $A(x)$ were obtained by counting the NECS with at most $8$ congruences.

Our main result is that the series $A$, the generating function for the number of NECS, is the reversion (compositional inverse) of the M\"obius series $M$. Of course, equivalently, this means that $M$ is the reversion of $A$.

\begin{theorem}\label{mainresult}
The series $A(x)$ is the unique solution to the functional equation
\[  M(A(x)) = x, \]
with initial condition $A(0) = 0$.
\end{theorem}

\begin{remark}
The coefficients of the reversion of the M\"obius series
are given by sequence \seqnum{A050385} in the
On-Line Encyclopedia of Integer Sequences \cite{sl}.
\label{rmk1}
\end{remark}

\begin{remark}
 We remark that, instead of the M{\"o}bius power series $M(x)$ defined in~\eqref{fpsAM}, it is more usual to study the Dirichlet series~$\sum_{k\ge 1} \mu(k) k^{-s}$. Indeed Hardy and Wright~\cite[p.~257]{hw60} refer to series such as~$M(x)$ as ``extremely difficult to handle''. The generating series~$M(e^{-y})$ was mentioned by Hardy and Littlewood~\cite[p.~122]{hl16}, and its unusual analytic properties were studied by Fr{\"o}berg
\cite{f66}.
\label{rmk2}
\end{remark}

\begin{remark}
A priori, it is not even obvious that the coefficients in the reversion of $M(x)$ are all positive. This fact follows from
our results.
\label{rmk3}
\end{remark}

We have not been able to use Theorem~\ref{mainresult} to determine a useful explicit expression for the number $a_k$ of NECS with~$k$ congruences. However, we are able to determine the precise asymptotic form for~$a_k$, as a corollary to Theorem~\ref{mainresult}.

\begin{theorem}\label{asymp}
Let~$\tau$ be the zero of $M'$ of smallest absolute value in $(-1,1)$, so that $M'(\tau)=0$. Also, let
\[ c= \sqrt{-\frac{M(\tau)}{2\pi M''(\tau)}} , \qquad\qquad\qquad \gamma = M(\tau)^{-1} \;  . \]
Then the $k$th coefficient in the reversion of the M\"obius series (which is also equal to the number of NECS with $k$ congruences) is asymptotically
\[ a_k \sim c \, \gamma^k \, k^{-3/2} . \]
\end{theorem}

\begin{remark}
Evaluations to $7$ decimal places of the constants in Theorem~\ref{asymp} are given by $\tau \doteq 0.3229939$, $c \doteq 0.0809423$ and $\gamma \doteq 5.4874522$.
\end{remark}

In Sections~\ref{sect2} and \ref{sect3} of this paper, we describe basic notation and terminology for ECS, NECS and for a set of rooted trees that arises in the study of NECS. The proof of Theorem~\ref{mainresult} is given in Section~\ref{sect4}.
Recurrences and numerical results appear in Section~\ref{sect5}. The proof of Theorem~\ref{asymp} is in Section~\ref{sect6}. In addition, we prove related results for the number of NECS with~$k$ congruences in which the gcd and lcm of the congruences are also specified.   In Section~\ref{sect7} we give some combinatorial results about the coefficients.   Some open problems are described in Section~\ref{sect8}.  We make some final remarks in 
Section~\ref{sect9}.

\section{Basic notation and definitions for exact covering systems}
\label{sect2}

\subsection{Exact covering systems}

For integers $n\ge 1$ and $0\le a < n$, let $\la a,n\ra $ denote the residue class $\{ x\in\bbZ \ : \ x\equiv a \ ( \mymod n ) \}$.   Let
$\sU \subseteq \bbZ$. The set of~$k\ge 1$ residue classes
\begin{equation}\label{NECSset}
 C = \{ \la a_i,n_i \ra \ : \  i=1, \ld ,k \}, 
\end{equation}
is called an {\em exact covering system} (ECS) of $\sU$ when
\begin{itemize}
\item the sets $\la a_i,n_i\ra $, $i = 1,\ld ,k$ are pairwise {\em disjoint}, and
\item the sets $\la a_i,n_i\ra$, $i =1,\ld ,k$ {\em cover} the set~$\sU$, i.e.,
\[ \dbigcup_{i=1}^k \;\la a_i,n_i \ra  =\sU, \]
\end{itemize}
where the symbol $\dcup$ indicates a disjoint union. Given an ECS $C$ as in~\eqref{NECSset}, suppose that $\gcd\{ n_i \ : i=1,\ld ,k\} = d$, and that $\lcm\{ n_i \ : \ i=1,\ld ,k\} = m$. Then we say that $C$ has {\em size} $k$, {\em gcd} $d$, and {\em lcm} $m$, written
\[ \vert C \vert = k,\qquad \quad \gcd (C) = d, \qquad \quad \lcm (C) = m. \]
In the case that~$\sU =\bbZ=\la 0,1\ra$, we say more simply that $C$ is an ECS (without mentioning the set~$\bbZ$). Table~\ref{tab1} lists all ECS of size at most~$4$, together with their gcd and lcm.

\begin{remark}
Note that the gcd of an ECS need not equal its smallest modulus, and the lcm need not equal its largest modulus.  The smallest counterexample to both of these claims is
$$ \{ \la 1, 4 \ra,
\la 3, 4 \ra,
\la 0,  6  \ra,
\la 2, 6 \ra,
\la 4, 6 \ra \}, $$
an ECS of size $5$, gcd $2$, and lcm $12$.
\end{remark}

\begin{table}[H]
\begin{center}
\begin{tabular}{|c|c|c|c|}
\hline
size & exact covering system & gcd & lcm \\
\hline
$1$ & $\{ \la 0,1 \ra  \}$ &  1 &  1\\
\hline
$2$ & $\{ \la 0,2 \ra, \la 1,2 \ra \}$ &  2 & 2\\
\hline
 & $ \{  \la 0,3 \ra , \la 1,3 \ra, \la 2,3 \ra \} $ & 3 & 3 \\
 $3$  & $\{  \la 0,2 \ra, \la 1,4 \ra , \la 3,4 \ra  \}$ &  2 & 4\\
  & $\{  \la 1,2 \ra , \la 0,4 \ra, \la 2,4 \ra \}$ &  2 & 4\\
\hline  
  & $ \{  \la 0,4 \ra , \la 1,4 \ra, \la 2,4 \ra, \la 3,4 \ra  \}$ &  4  & 4\\
& $\{  \la 0,2 \ra, \la 1,6 \ra, \la 3,6 \ra, \la 5,6 \ra \}$ & 2 & 6\\
&  $\{  \la 0,3 \ra, \la 1,3 \ra, \la 2,6 \ra , \la 5,6\ra \}$ &  3 & 6\\
&  $\{  \la 0,3 \ra , \la 2,3 \ra, \la 1,6 \ra, \la 4,6 \ra \}$ &  3 & 6\\
$4$ &  $\{  \la 1,2 \ra,  \la 0,6 \ra, \la 2,6 \ra, \la 4,6 \ra \}$ &  2 & 6\\
&  $\{  \la 1,3 \ra,  \la 2,3 \ra ,  \la 0,6 \ra, \la 3,6 \ra  \} $ & 3 & 6\\
&  $\{  \la 0,2 \ra, \la 1,4 \ra,  \la 3,8 \ra, \la 7,8 \ra \}$ &  2 & 8\\
&  $\{ \la 0,2 \ra,  \la 3,4 \ra, \la 1,8 \ra, \la 5,8 \ra \}$ &  2 & 8\\
&  $\{  \la 1,2 \ra,  \la 0,4 \ra,  \la 2,8 \ra, \la 6,8 \ra \}$ & 2 & 8\\
&  $\{  \la 1,2 \ra,  \la 2,4 \ra, \la 0,8 \ra, \la 4,8 \ra \}$ &  2 & 8\\
\hline
\end{tabular}
\end{center}

\vspace{.2in}

\caption{The ECS of size at most $4$.}\label{tab1}
\end{table}

We now define two constructions for ECS.  First, consider an ECS 
$$C = \{ \la a_i,n_i \ra \ : \  i=1, \ld ,k \},$$ and a residue class~$\la b,c\ra $.
Define
\begin{equation}\label{bcexpn}
E_{\la b,c\ra}(C) =  \{ \la b + c \, a_i,c \, n_i \ra \ : \  i=1, \ld ,k \}, 
\end{equation}
which we refer to as the~$\la b,c\ra$-{\em expansion} of~$C$. Note that $E_{\la b,c\ra}(C)$ is itself an ECS of~$\la b,c\ra$.

Second, consider an ECS $C$, a residue class~$\la a,n\ra \in C$, and an integer $r\ge 2$. Let $C'$ be the set of residue classes obtained by removing $\la a,n\ra$ from~$C$ and replacing it by the~$r$ residue classes
\begin{equation}\label{rsplit}
\la a+jn,rn\ra, \qquad\qquad j=0,1,\ld ,r-1.
\end{equation}
Note that the residue classes in~\eqref{rsplit} are pairwise disjoint, and that they contain between them all integers in~$\la a,n\ra$; indeed, they are an ECS of the residue class~$\la a,n\ra$ of size~$r$ . An immediate consequence of this is that~$C'$ is also an ECS, with size given by~$\vert C'\vert = \vert C \vert +r-1$. We say that the collection of classes in~\eqref{rsplit} is the $r$-{\em split} of~$\la a,n\ra$, and that~$C'$ is an $r$-split of~$C$. Equivalently, we say that the collection of classes in~\eqref{rsplit} is obtained by~$r$-{\em splitting}~$\la a,n\ra$, and that~$C'$ is obtained by $r$-splitting~$C$. We also use the terms {\em split} and {\em splitting} in these same contexts, when we don't choose to specify the value of $r\ge 2$.

\subsection{Natural exact covering systems}\label{necs}

Let~$\sA$ be the set of exact covering systems that can be obtained by a finite sequence (possibly empty) of splits applied to~$\{ \la 0,1\ra\}=\{\bbZ\}$; we call this a {\em split sequence}. It is important to note that if this split sequence is an $r_1$-split, an $r_2$-split, $\ld$ , an $r_m$-split, for some~$m\ge 0$, then the values of $r_1,\ld ,r_m$ need not be the same, and indeed can vary arbitrarily over $r_1,\ld ,r_m \ge 2$ (and $m$ can be any non-negative integer).

The exact covering systems in~$\sA$ are called {\em natural exact covering systems}~(NECS). It is well known that in general (see, e.g.,~\cite[p.~392]{k84}) elements of~$\sA$ can be obtained by a split sequence in more than one way. As an easy example of this, the NECS of size $6$
\begin{equation}\label{4two2}
\{  \la 0, 6 \ra,
\la 1, 6 \ra,
\la 2,  6  \ra,
\la 3, 6 \ra,
\la 4,  6  \ra,
\la 5, 6 \ra  \}
\end{equation}
can be obtained by~$6$-splitting~$\{ \la 0,1\ra\}$, but it can also be created by the following split sequence of length~$3$: first~$2$-split~$\{ \la 0,1\ra\}$, then~$3$-split~$\{ \la 0,2\ra\}$, then~$3$-split~$\{ \la 1,2\ra\}$.

It has also long been known that not every ECS is an NECS, e.g., Porubsk{\'y}~\cite{p74}. The smallest size for an ECS that is not an NECS is~$13$,~\cite[Example 3.1]{s15}, which might seem surprisingly large when one first encounters the study of ECS. In particular, this means that all of the ECS in Table~\ref{tab1} are also NECS (equivalently, the caption for Table~\ref{tab1} could also be ``The NECS of size at most~$4$.'').

\section{Rooted ordered trees and NECS}
\label{sect3}

\subsection{Rooted ordered trees}

Let~$\sT$ be the set of rooted ordered trees with a finite (nonempty) set of vertices, in which each vertex has~$r$ ordered {\em children}, for some~$r$ in~$\{ 0,2,3,4,\ld \}$. The rooted tree consisting of a single (root vertex) is in~$\sT$, and we denote this rooted tree by~$\vep$.  We regard the trees in~$\sT$ as being embedded in the plane, with the root vertex at the bottom, and the children of each vertex above that vertex, ordered from left to right. Hence we refer to a vertex with~$r$ children as a vertex of {\em up-degree}~$r$, where~$r=0$ or~$r\ge 2$. We denote the set of trees in which the up-degree of the root vertex is~$r$ by~$\sT^{(r)}$, so we have~$\sT=\sT^{(0)} \dcup \sT^{(2)}\dcup \sT^{(3)}\dcup \cdots$, and~$\sT^{(0)}=\{ \vep\}$.

A vertex of up-degree~$0$ in a tree is called a {\em leaf}. Thus the root vertex in the tree~$\vep$ is a leaf (even though it has degree~$0$ in the graph sense). For a tree~$T$ in~$\sT$ rooted at vertex~$w$, the {\em height} of~$w$ in~$T$ is~$0$, and the height of any other vertex~$v$ in~$T$ is the edge-length of the unique path in~$T$ from the root~$w$ to~$v$. The {\em height} of the tree~$T$ itself, denoted by~$\het(T)$, is equal to the maximum height among the leaves in~$T$. Also, the number of leaves in~$T$ is denoted by~$\lam(T)$. For example, we have~$\het(\vep)=0$ and~$\lam(\vep)=1$, since the single vertex in~$\vep$ is a leaf at height~$0$.

For vertices~$u\neq v$ in tree~$T\in\sT$, we say that~$u$ is a {\em descendant} of~$v$ when either~$u$ is a child of~$v$, or is (recursively) the descendant of any child of~$v$. We denote the subtree of~$T$ whose vertices consist of~$v$ and its descendants by~$T_v$. For any vertex~$v$ of~$T\in\sT$, we have~$T_v\in\sT$. In particular, if~$v$ has up-degree~$r$ in~$T$, then~$T\in\sT^{(r)}$. Finally, if~$v$ is the root vertex of~$T$, then we have~$T_v=T$.

In Figure~\ref{treepic} we display an example of a tree~$T$ in~$\sT$ with root vertex~$w$. The root vertex has up-degree~$3$ (so~$T\in\sT^{(3)}$), and its three children, in order, are vertices~$x,y,z$. The tree has~$\lam(T)=10$ leaves;~$1$ (namely $y$) at height~$1$,~$3$ at height~$2$,~$4$ at height~$3$, and~$2$ at height~$4$. Hence the height of the tree is~$\het(T)=4$. For the subtrees~$T_x$,~$T_y$,~$T_z$ of~$T$ rooted at~$x,y,z$, respectively, we have~$\lam(T_x)=4$,~$\het(T_x)=2$, $T_y=\vep$ (so $\het(T_y)=0$,~$\lam(T_y)=1$), and $\lam(T_z)=5$,~$\het(T_z)=3$.

\begin{figure}[H]
\begin{center}
\includegraphics[width=1.7in]{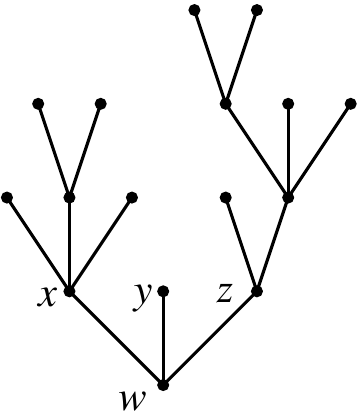}\ 
\hspace{.3in} \ 
\includegraphics[width=2.25in]{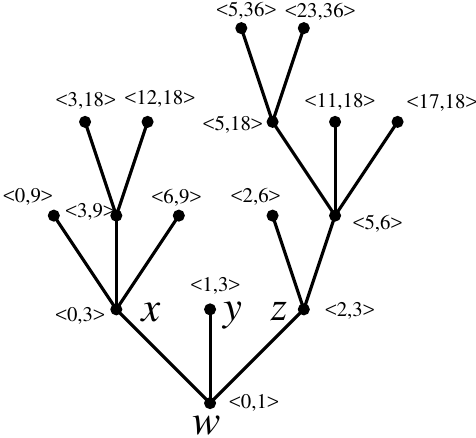}
\end{center}
\caption{Left:  a rooted ordered tree in the set~$\sT$.  Right: the same tree with~$\rho$ values assigned to vertices as in~\S 3.2.}\label{treepic}
\end{figure}

\subsection{The relationship between trees and NECS}

The reason for introducing the set of trees~$\sT$ is that every NECS can be represented by at least one tree in~$\sT$. To see this, consider a tree~$T$ in~$\sT$, and for each vertex~$v$ in~$\sT$, assign a residue class~$\rho(v)$ (which we'll refer to as ``assigning a ~$\rho$ value'' to~$v$), using the following iterative assignment procedure:
\begin{itemize}
\item Initially, assign~$\rho(w)=\la 0,1\ra$, where~$w$ is the root vertex of~$T$;
\item At every stage, find a non-leaf vertex~$u$ of~$T$ that has been assigned a~$\rho$ value, but whose children have not, and suppose that~$\rho(u)=\la a,n\ra$, for some~$0\le a<n$. Then, denoting the children of $u$ by $c_1,\ld ,c_r$, in order (i.e.,~$u$ has up-degree~$r$ for some~$r\ge 2$), assign
\[ \rho(c_{j+1}) = \la a+jn,rn \ra ,\qquad\qquad j=0,1,\ld ,r-1. \]
Stop when there is no such vertex~$u$.
\end{itemize}

Compare the assignments~$\rho(c_{j+1})$ given above for the children of a vertex with up-degree~$r$ to the~$r$-split of residue classes defined in~\eqref{rsplit}. Clearly, the successive stages for assigning~$\rho$ values to vertices of a tree in~$\sT$ in the iterative assignment procedure above, starting with assigning the~$\rho$ value~$\la 0,1\ra$ to the root,  is a realization of a split sequence (one split for each vertex that is not a leaf) applied to~$\la 0,1\ra$, and hence corresponds to an element of~$\sA$. The order of the splits in the split sequence corresponds to any total order of the set of non-root, non-leaf vertices in the tree that is a linear extension of the following partial order:~$u\prec v$ when~$v$ is a descendant of~$u$.

Recall that when we apply an~$r$-split in an element of~$\sA$, we replace one residue class by a set of~$r$ residue classes. Hence, in the corresponding tree in~$\sT$, when we assign a label to a vertex~$u$, we want the residue class~$\rho(u)$ to disappear, and to be replaced by the residue classes of its children. This means that if the leaves in a tree~$T\in\sT$ are given by~$\ell_1,\ldots,\ell_k$, $k\ge 1$, then the NECS corresponding to~$T$ is exactly the set of leaf labels~$\{\rho(\ell_1),\ldots,\rho(\ell_k)\}$, and we write
\[ \chi(T) = \{ \rho(\ell_1),\ldots ,\rho(\ell_k)\} . \]

For example, for the tree~$T$ in~$\sT$ displayed on the left side of Figure~\ref{treepic}, using the assignment procedure we initially assign~$\rho(w)=\la 0,1\ra$, then at the first stage we assign~$\rho(x) = \la 0,3\ra$, $\rho(y) = \la 1,3\ra$, $\rho(z) = \la 2,3\ra$. Upon completion of all stages of the assignment procedure, each vertex of~$T$ has a~$\rho$ value, which is displayed beside the vertex on the right side of Figure~\ref{treepic}. We thus obtain  
\begin{equation}\label{chiexample}
\chi(T) = \{ \la 0, 9 \ra,
\la 3, 18 \ra,
\la 12,  18  \ra,
\la 6, 9 \ra,
\la 1, 3 \ra,
\la 2, 6 \ra,
\la 5, 36 \ra,
\la 23,  36  \ra,
\la 11, 18 \ra,
\la 17, 18 \ra \},
\end{equation}
where this set of~$10$ residue classes gives the assigned~$\rho$ values for the~$10$ leaves in~$T$.

The correspondence that we have denoted by~$\chi$, between the set of rooted trees in~$\sT$ and the set~$\sA$ of NECS, is standard, and has appeared in the literature on NECS from the beginning, e.g., Porubsk{\'y}~\cite{p74} and Zn{\'a}m~\cite{z82}. In~\cite{p74} each vertex is identified with the residue class that we ``assign'' in our description above, and in~\cite{z82} the trees are treated in a slightly different but equivalent way using the notion of {\em product-distance}, and are referred to as~$\bbZ$-{\em trees}.

It is also standard that this correspondence between~$\sT$ and~$\sA$ is not one-to-one; this is simply a restatement of the fact we mentioned in Section~\ref{necs} above, that in general the elements of~$\sA$ can be obtained by a split sequence in more than one way. We summarize this situation in the following result.
\begin{proposition}\label{chionto}
The function
\[ \chi \colon \sT \to \sA \colon T\mapsto C    \]
is a surjection, in which $\lam(T)=\vert C\vert$.
\end{proposition}

\begin{remark}
Proposition~\ref{chionto} makes it clear that the NECS of size~$k$ correspond to certain equivalence classes of trees in~$\sT$ with~$k$ leaves. The problem of counting the total number of trees in~$\sT$ with~$k$ leaves is well known. Let~$t_k$ denote the number of trees in~$\sT$ with~$k$ leaves, $k\ge 1$, and $T(x)  = \sum_{k\ge 1} t_k x^k$. Then the~$t_k$ are called {\em Schr\"oder} numbers, and the generating function has the closed form
\[ T(x) = \frac{1}{4}\Big( 1+x - \sqrt{1-6x+x^2}   \Big) ,  \]
see, e.g.,~\cite[pp.~69--70]{fs09} for a detailed description. An asymptotic form for~$t_k$ also appears in~\cite[pp.~474--475]{fs09} (as corrected at \url{http://ac.cs.princeton.edu/errata/}): 
\[ t_k \sim \omega \, \big( 3+2\sqrt{2} \big)^k \, k^{-3/2},\qquad\qquad \omega = \frac{1}{2^{\frac{7}{4}}\;\sqrt{\pi \big( 3+2\sqrt{2} \big) }} . \]
Of course,~$t_k$ is an upper bound for~$a_k$, reflecting the fact that~$3+2\sqrt{2}\doteq 5.828$, the asymptotic growth rate for~$t_k$, is larger than~$\gamma\doteq 5.487$, the asymptotic growth rate for~$a_k$ appearing in Theorem~\ref{asymp}.
\label{schro}
\end{remark}

\subsection{Subtrees rooted at children of the root and NECS}

For~$n\ge 2$ and~$T\in \sT^{(n)}$, suppose that the children of the root vertex of~$T$ are~$x_1,\ld ,x_n$, in order. When we apply our iterative assignment procedure to~$T$, we obtain~$\rho(x_i)=\la i-1,n\ra$, for $i=1,\ld ,n$. Then the residue classes that are assigned to the leaves in~$T_{x_i}$ form an NECS of~$\la i-1,n\ra$, and it is easy to check that this NECS is simply~$E_{\la i-1,n \ra}(\chi(T_{x_i}))$, using the expansion notation defined in~\eqref{bcexpn}. Putting the residue classes for all subtrees together, we thus obtain the relationship
\begin{equation}\label{subchi}
\chi(T)= \dbigcup_{i=1}^n E_{\la i-1,n \ra}(\chi(T_{x_i})).
\end{equation}

For example, for the tree~$T$ in~$\sT$ in Figure~\ref{treepic}, the root vertex~$w$ has up-degree~$3$, and the children of the root are~$x,y,z$, in order. For the subtrees rooted at~$x,y,z$, we have
\begin{align*} \chi(T_x) &= \{ \la 0, 3 \ra, \la 1, 6 \ra, \la 4,  6  \ra, \la 2, 3 \ra \}, \\
\chi(T_y) &= \{ \la 0,1 \ra \}, \\
\chi(T_z) &= \{ \la 0, 2 \ra, \la 1, 12 \ra, \la 7,  12  \ra, \la 3, 6 \ra, \la 5, 6 \ra \},   
\end{align*}
and the appropriate expansions of these NECS are given by
\begin{align*} E_{\la 0,3 \ra}(\chi(T_x)) &= \{ \la 0, 9 \ra, \la 3, 18 \ra, \la 12,  18  \ra, \la 6, 9 \ra \}, \\
E_{\la 1,3 \ra}(\chi(T_y)) &= \{ \la 1,3 \ra \}, \\
E_{\la 2,3 \ra}(\chi(T_z)) &= \{ \la 2, 6 \ra, \la 5, 36 \ra, \la 23,  36  \ra, \la 11, 18 \ra, \la 17, 18 \ra \}.   
\end{align*}
Comparing these with~\eqref{chiexample}, we have
\[ \chi(T) = E_{\la 0,3 \ra}(\chi(T_x)) \; \dcup \; E_{\la 1,3 \ra}(\chi(T_y)) \; \dcup \; E_{\la 2,3 \ra}(\chi(T_z)),  \]
confirming that relationship~\eqref{subchi} holds for the tree~$T$ in Figure~\ref{treepic}.

In the following results, we record some useful properties for subtrees of children of the root that follow immediately from~\eqref{subchi}.

\begin{proposition}\label{subtreedistn}
For~$T\in \sT^{(n)}$, $n\ge 2$, suppose that the children of the root vertex of~$T$ are~$x_1,\ld ,x_n$, in order, and that we have
\[ \chi(T)=C,\qquad \chi(T_{x_i})= C_i,, \quad i=1,\ld ,n.   \]
Then
\begin{enumerate}[(a)]
\item $\quad \vert C\vert = \vert C_1\vert +\cdots + \vert C_n\vert$ ,
\item $\quad \gcd(C) = n\cdot\gcd\{ \gcd(C_1), \ld , \gcd(C_n)\}$ ,
\item $\quad \lcm(C) = n\cdot\lcm\{ \lcm(C_1), \ld , \lcm(C_n)\}$ .
\end{enumerate}
\end{proposition}

\begin{proposition}\label{subtreeequiv}
For~$P,Q\in\sT^{(n)}$, $n\ge 2$, suppose that the children of the root vertex of~$P$ (respectively,~$Q$) are~$y_1,\ld ,y_n$ (respectively, $z_1,\ld ,z_n$), in order. Then $\chi(P)=\chi(Q)$ if and only if $\chi(P_{y_i})=\chi(Q_{z_i})$, $i=1,\ld ,n$.
\end{proposition}

Now we turn to a different type of result for rooted trees~$T$, in which we give a bijective construction that preserves the corresponding NECS~$\chi(T)$. For compactness in stating the result, let~$\sG_{a,b}$ denote the set of trees~$T\in\sT^{(a)}$ such that all children of the root vertex have up-degree $b$, ~$a,b\ge 2$. 

\begin{lemma}\label{ablemma}
For $a,b\ge 2$, there is a bijection
\[ \sT^{(ab)} \to \sG_{a,b} \colon T \mapsto S \]
with~$\chi(T)=\chi(S)$.
\end{lemma}

\begin{proof}
Consider~$T\in\sT^{(ab)}$, and let the children of the root vertex~$w$ of~$T$ be~$x_1,\ld ,x_{ab}$, in order. We construct a tree~$S\in\sG_{a,b}$ as follows: The root vertex~$u$ of~$S$, and its children~$y_1,\ld ,y_a$, in order, are newly created vertices (i.e., they are not vertices in~$T$). For $i=1,\ld ,a$, vertex~$y_i$ at height~$1$ in~$S$ has~$b$ children, given by vertices $x_{i},x_{i+a}, \ld ,x_{i+(b-1)a}$ of~$T$, in order. For $m=1,\ld ab$, the construction of~$S$ is completed by rooting subtree~$T_{x_m}$ of~$T$ at vertex~$x_m$ of~$S$ (to become the subtree~$S_{x_m}$ of~$S$).

Note that when we apply our iterative assignment procedure to~$T$, we obtain~$\rho(x_m)=\la m-1,ab \ra$, for~$m=1,\ld ,ab$ (in which $x_m$ is regarded as a vertex in~$T$). Also, it is easy to check that when we apply our iterative assignment procedure to~$S$, we obtain~$\rho(x_m)=\la m-1,ab \ra$, for~$m=1,\ld ,ab$ (in which~$x_m$ is regarded as a vertex in~$S$). But this implies that all vertices in the subtrees~$T_{x_m}=S_{x_m}$, $m=1,\ld ,ab$, will be assigned the same~$\rho$ values, and hence that~$\chi(T)=\chi(S)$.

This construction is clearly reversible, and thus gives the required bijection.
\end{proof}

We now apply Lemma~\ref{ablemma} to prove a result that will be key for our proof of the main result in Section~\ref{sect4}. Again for compactness in stating the result, let~$\sD_n$ denote the set of trees~$T\in\sT$ such that $n \, \vert \, \gcd(\chi(T))$, $n\ge 2$. 

\begin{proposition}\label{treedivprop}
For~$n\ge 2$,
\[ \{ \chi(S) \colon S \in \sD_n \} = \{ \chi(T) \colon T \in \sT^{(n)} \}. \]
\end{proposition}

\begin{proof}
Let~$\Omega_1 = \{ \chi(S) \colon S \in \sD_n \}$ and~$\Omega_2= \{ \chi(T) \colon T \in \sT^{(n)} \}$. We will prove that~$\Omega_1=\Omega_2$ by proving both containments~$\Omega_1\subseteq\Omega_2$ and~$\Omega_2\subseteq\Omega_1$.

First, to prove~$\Omega_2\subseteq\Omega_1$. Proposition~\ref{subtreedistn}(b) implies that~$\sT^{(n)}\subseteq \sD_n$, from which~$\Omega_2\subseteq\Omega_1$ follows immediately.

Second, to prove~$\Omega_1\subseteq\Omega_2$. It is sufficient to prove the following implication: For all~$n\ge 2$ and~$S\in\sD_n$, there exists~$T\in\sT^{(n)}$ such that~$\chi(T)=\chi(S)$. We prove this by induction on the height of~$S$. For the base case, consider~$S\in\sD_n$ of height equal to~$1$. Hence, from Proposition~\ref{subtreedistn}(b),~$S\in\sT^{(nb)}$ for some $b\ge 1$, in which the~$nb$ children of the root in~$S$ are leaves. If~$b=1$, then~$S\in\sT^{(n)}$, giving the result immediately. If~$b\ge 2$, then Lemma~\ref{ablemma}  with~$a=n$ implies that there exists a tree~$R\in\sG_{n,b}\subseteq\sT^{(n)}$ with~$\chi(R)=\chi(S)$, proving the result in this case. 

Assume that the implication is true for all~$n\ge 2$ and trees in~$\sD_n$ of height at most~$k$, for some~$k\ge 1$. Consider~$n\ge 2$ and an arbitrary tree~$S\in\sD_n$ with~$\het(S)=k+1$. Using Proposition~\ref{subtreedistn}(b), there are three cases:
\begin{itemize}
\item $S\in\sT^{(n)}$, which gives the result immediately. 
\item $S\in\sT^{(nb)}$ for some $b\ge 2$. Then Lemma~\ref{ablemma} with~$a=n$ implies that there exists a tree~$R\in\sG_{n,b}\subseteq\sT^{(n)}$ with~$\chi(R)=\chi(S)$, proving the result in this case. 
\item $S\in\sT^{(a)}$ for some proper divisor~$a\ge 2$ of~$n$, so~$n=ab$,~$a,b\ge 2$. In this case, in addition, if~$x_1,\ld,x_a$ are the children of the root vertex of~$S$, then from Proposition~\ref{subtreedistn}(b) we have
\[ b \, \vert \, \gcd\{\gcd(\chi(S_{x_1})),\ld , \gcd(\chi(S_{x_a})) \},  \]
and hence~$b \, \vert \, \gcd(\chi(S_{x_i}))$ for $i=1,\ld ,a$. Equivalently,~$S_{x_i}\in\sD_b$ for~$i=1,\ld ,a$. But~$\het(S_{x_i})\le k$ for~$i=1,\ld ,k$. Hence, for~$i=1,\ld ,k$, from the induction hypothesis, there exists~$R_{x_i}\in\sT^{(b)}$ (also rooted at vertex~$x_i$) such that $\chi(R_{x_i})=\chi(S_{x_i})$. Now construct a tree~$R$ by removing the subtree~$S_{x_i}$ from~$S$, and replacing it by the subtree~$R_{x_i}$, for~$i=1,\ld ,a$. Note that~$R\in\sG_{a,b}$, and that~$\chi(R)=\chi(S)$, from Proposition~\ref{subtreeequiv}. Then, Lemma~\ref{ablemma} implies that there exists a tree~$Q\in\sT^{(ab)}=\sT^{(n)}$ with~$\chi(Q)=\chi(R)=\chi(S)$, proving the result in this case.
\end{itemize}
This completes the inductive proof of the implication, and thus that~$\Omega_1\subseteq\Omega_2$.
\end{proof}

\section{Proof of the main result}
\label{sect4}

It will be convenient to partition the set~$\sA$ of NECS according to gcd. Hence let $\sA_{m}$ denote the set of NECS with gcd~$m$, for $m\ge 1$. Recall that~$a_k$ is the number of NECS of size~$k$, for~$k\ge 1$, and let~$a_{k,m}$ denote the number of NECS of size~$k$ and gcd~$m$, for~$k,m\ge 1$. As in~\eqref{fpsAM}, we have the generating function
\begin{equation}\label{agenser}
A(x) = \sum_{k\ge 1} a_k x^k = \sum_{C\in\sA} x^{\vert C\vert} ,
\end{equation}
and we define the additional generating functions
\begin{equation}\label{AAjmdefn} 
 A_m(x) = \sum_{k\ge 1} a_{k,m} x^k = \sum_{ C\in\sA_m } x^{\vert C\vert} ,\qquad m\ge 1. 
 \end{equation}
Of course, these generating functions are related by
\begin{equation}\label{AAjdefn} 
A(x) = \sum_{m\ge 1} A_m(x) .
\end{equation}

Note that the situation for~$C\in\sA$ with $\gcd(C)=1$ is particularly simple: We must have~$C=\chi(T)$ for some~$T\in \sT^{(0)}$, from Proposition~\ref{subtreedistn}(b). But the only tree in~$\sT^{(0)}$ is the single-vertex tree~$\vep$, and we have~$\chi(\vep) = \la 0,1\ra$. Since~$\la 0,1\ra$ has both size and gcd equal to~$1$, we conclude that~$\sA_1=\{ \la 0,1 \ra \}$, and hence from~\eqref{AAjmdefn} that
\begin{equation}\label{A1x}
A_1(x)=x .
\end{equation}

In part~(a) of the following result, we prove a functional equation for the generating functions~$A(x)$ and~$A_m(x), m\ge 1$, that is a generalization of equation~\eqref{AAjdefn} above. The proof that we give for part~(b) of the result is to apply M\"obius inversion to part~(a), which is the reason that the M\"obius series~$M$ defined in~\eqref{fpsAM} appears in the statement of part~(b).

\begin{theorem}\label{mainmthm}
\leavevmode
\begin{enumerate}[(a)]
\item
For $n\ge 1$ we have
\[ A(x)^n = \sum_{d \ge 1} A_{nd}(x).  \]
\item
For $m\ge 1$ we have
\[ M\big( A(x)^m \big) = A_m(x). \]
\end{enumerate}
\end{theorem}

\begin{proof}
\begin{enumerate}[(a)]
\item
For $n=1$, the result is given by~\eqref{AAjdefn} above.

For $n\ge 2$, we begin the proof by defining~$\sU(\sA,n) = \dcup_{d\ge 1} \sA_{nd}$, the set of all~$C\in\sA$ such that~$n \mid \gcd(C)$. Now, from Proposition~\ref{treedivprop} and Proposition~\ref{chionto}, we have
\[ \sU(\sA,n) = \{ \chi(T)\colon T\in \sT^{(n)}  \}.  \]
But removing the root vertex from a tree in~$\sT^{(n)}$ to obtain an ordered list of $n$ rooted trees in $\sT$ (namely, the subtrees rooted at the~$n$ children of the root vertex), yields the usual bijection between~$\sT^{(n)}$ and~$\sT^n$ (the set of~$n$-tuples of elements of $\sT$). Together with Proposition~\ref{subtreeequiv}, as well as Proposition~\ref{chionto}, this implies that there is a bijection between~$\{ \chi(T)\colon T\in \sT^{(n)}  \}$ and the set~$\sA^n$ of~$n$-tuples of NECS in~$\sA$. Putting these pieces together and eliminating the set~$\{ \chi(T)\colon T\in \sT^{(n)} \}$ yields a bijection
\begin{equation}\label{biject}
\sU(\sA,n) \to \sA^n \colon C \mapsto (C_1,\ld ,C_n).
\end{equation}
Moreover, in the above bijection, from Proposition~\ref{subtreedistn}, $C$ and its image~$(C_1,\ld ,C_n)$ are related by the equations
\begin{align}
\vert C\vert &= \vert C_1\vert +\ld + \vert C_n\vert , \label{bijsize}\\
\gcd(C) &= n\cdot\gcd\{ \gcd(C_1), \ld , \gcd(C_n)\} ,  \label{bijgcd} \\
\lcm(C) &= n\cdot\lcm\{ \lcm(C_1), \ld , \lcm(C_n)\}. \label{bijlcm}
\end{align} 

Now we turn to generating functions. Applying bijection~\eqref{biject} to the range of summation below, and using relation~\eqref{bijsize}, we obtain
\[  \sum_{C \in \sU(\sA,n)} x^{\vert C\vert} = \sum_{(C_1,\ld ,C_n)\in\sA^n} x^{ \vert C_1\vert +\ld + \vert C_n\vert }. \]
But
\[ \sum_{C \in \sU(\sA,n)} x^{\vert C\vert} = \sum_{d\ge 1} \sum_{C \in \sA_{nd}} x^{\vert C\vert} =  \sum_{d \ge 1} A_{nd}(x) ,  \]
from~\eqref{AAjmdefn}, and
\[ \sum_{(C_1,\ld ,C_n)\in\sA^n} x^{ \vert C_1\vert +\cdots + \vert C_n\vert } =\prod_{i=1}^n \sum_{C_i\in\sA}x^{\vert C_i\vert} = A(x)^n,  \]
from~\eqref{agenser}, completing the proof of the result for $n\ge 2$.

\item
For fixed $m\ge 1$, replace $n$ in part~(a) of this result by $mn$, multiply on both sides by $\mu(n)$, and sum over $n\ge 1$, to obtain
 \[ \sum_{n\ge 1}\mu(n)A(x)^{mn} = \sum_{n\ge 1} \mu(n) \sum_{d\ge 1} A_{mnd}(x) .  \]
But the right-hand side of this equation can be rewritten as
\[ \sum_{n\ge 1} \mu(n) \sum_{d\ge 1} A_{mnd}(x) = \sum_{i\ge 1} A_{mi}(x) \sum_{n\vert i } \mu(n)= A_m(x), \]
using the standard fact (see, e.g.,~\cite[p.~235]{hw60}) that
\begin{equation}\label{mu10}
\sum_{n\vert i} \mu(n)=\begin{cases}
\, 1,\qquad & \text{if $i=1$;} \\
\, 0,\qquad & \text{if $i\ge 2$.}
\end{cases}
\end{equation}
Hence we obtain the equation
\begin{equation} 
\sum_{n\ge 1}\mu(n)A(x)^{mn} = A_m(x), 
\label{niceq}
\end{equation}
or $M\big( A(x)^m\big) =A_m(x)$, as required.
\end{enumerate}
\end{proof}

We are now able to prove the main result, as a simple consequence of the above Theorem.

\begin{proof}[Proof of Theorem~\ref{mainresult}]
Specializing Theorem~\ref{mainmthm}(b) to the case~$m=1$ gives~$M(A(x)) = A_1(x)$.
The result follows immediately from~\eqref{A1x}.
\end{proof}

\section{Recurrences specifying size, gcd and lcm}
\label{sect5}

In the following result we give a recurrence equation for the numbers~$a_{k,m}$ of NECS with size~$k$ and gcd~$m$.

\begin{proposition}\label{maincoeff}
For $k,n,d\ge 1$,
\begin{equation}\label{conjR}
 \sum a_{j_1,m_1} \cdots a_{j_n,m_n} = a_{k,nd}, 
 \end{equation}
where the summation on the left-hand side is over $j_1,\ldots ,j_n\ge 1$ and $m_1,\ldots ,m_n\ge 1$ such that
\begin{equation}\label{bijcond}
  j_1+ \cdots +j_n=k, \qquad\quad \textit{and} \qquad\quad \gcd\{ m_1, \ldots , m_n\} =d  .
  \end{equation}
 \end{proposition}
 
 \begin{proof}
 For~$n=1$, the result is simply the identity~$a_{k,d}=a_{k,d}$.
 
 For~$n\ge 2$, consider bijection~\eqref{biject}. The number of elements in the set~$\sU(\sA,n)$ with size~$k$ and~gcd~$nd$ is give by~$a_{k,nd}$. But from the bijection, this is equal to the number of $n$-tuples $(C_1,\ld ,C_n)\in \sA^n$ in which, from~\eqref{bijsize} and~\eqref{bijgcd}, we have
\[  \vert C_1\vert +\cdots + \vert C_n\vert = k,\qquad\text{and}\qquad \gcd\{ \gcd(C_1), \ld , \gcd(C_n)\} = d  . \]
The result for~$n\ge 2$ follows immediately.
\end{proof}

\begin{remark}
Proposition~\ref{maincoeff}, with $d = 1$, can be used to list all elements of~$\sA_{k,n}$, the set of NECS with size $k$ and gcd $n$, and also to count their number~$a_{k,n}$, using either dynamic programming, or recursion together with memoization.  We briefly describe this second approach.    The base cases of the recursion are
$k = n$ (for which the only NECS is
$\{ \la 0,k \ra, \la 1, k \ra, \ldots, \la k-1, k \ra \}$) and $n = 1$ (for which the only NECS is
$\{ \la 0,1 \ra \}$ corresponding to $k = 1$).  Given $k$ and $n$ as input, we can easily compute all  
$\binom{k-1}{n-1}$ compositions of $k$ into 
$n$ positive parts (using, for example,
the algorithm in \cite[Chap.~5]{nw78}).  We now discard those compositions whose $\gcd$ is greater than one.  For each composition $(j_1, j_2, \ldots, j_n)$ that remains, we
consider all $j_1 j_2 \cdots j_n$ of the
$n$-tuples $(m_1, m_2, \ldots, m_n)$
satisfying $1 \leq m_i \leq j_i$ for $i = 1,2,\ldots n$.
For each element $(j_i, m_i)$ in the list of pairs
$((j_1, m_1), \ldots, (j_n, m_n))$ we
recursively compute all the NECS~$C_i$ in~$\sA_{j_i, m_i}$. Using the expansion construction given in~\eqref{bcexpn}, we finally create the NECS
\[ \dbigcup_{i=1}^n E_{\la i-1 ,n \ra } (C_i)  . \]
If we are only interested in counting these NECS, we sum all the products
$a_{j_1,m_1} \cdots a_{j_n, m_n}$ instead.

Using an implementation of this algorithm written in
APL, we computed $a_{k,n}$ for $1 \leq n \leq k \leq 22$.  We report
the results for $1 \leq n \leq k \leq 13$ in Table~\ref{tabsizegcd13}.  The code is available at
\centerline{\url{https://cs.uwaterloo.ca/~shallit/papers.html} \ .}
\end{remark}

\begin{table}[H]
\begin{center}
\begin{tabular}{c|ccccccccccccccc}
$k \backslash n$ & 1 & 2 & 3 & 4 & 5 & 6 & 7 & 8 & 9 & 10 & 11 & 12 & 13 \\
\hline
1 & 1 \\
2 & 0 & 1 \\
3 & 0 & 2 & 1 \\
4 & 0 & 6 & 3 & 1 \\
5 & 0 & 22 & 12 & 4 & 1 \\
6 & 0 & 88 & 48 & 18 & 5 & 1 \\
7 & 0 & 372 & 207 & 80 & 25 & 6 & 1 \\
8 & 0 & 1636 & 918 & 366 & 120 & 33 & 7 & 1 \\
9 & 0 & 7406 & 4188 & 1700 & 580 & 170 & 42 & 8 & 1 \\
10 & 0 & 34276 & 19488 & 8026 & 2810 & 864 & 231 & 52 & 9 & 1 \\
11 & 0 & 161436 & 92199 & 38384 & 13710 & 4356 & 1232 & 304 & 63 & 10 & 1 \\
12 & 0 & 771238 & 442056 & 185644 & 67330 & 21936 & 6454 & 1698 & 390 & 75 & 11 & 1 \\
13 & 0 & 3728168 & 2143329 & 906472 & 332825 & 110562 & 33523 & 9232 & 2277 & 490 & 88 & 12 & 1 \\
\end{tabular}
\end{center}
\caption{Table of values for $a_{k,n}$, $1\le k,n\le 13$}\label{tabsizegcd13}
\end{table}

Now let $a_{k,m,\ell}$ denote the number of NECS of size $k$, gcd $m$ and lcm $\ell$, for $k,m,\ell\ge 1$. Of course, $a_{k,m,\ell} = 0$ unless $m\,\vert\, \ell$. The following is a version of Proposition~\ref{maincoeff} that records the lcm as well as size and gcd.  It can be proved in the same way, by considering the bijection~\eqref{biject} together with relations~\eqref{bijsize},~\eqref{bijgcd} and~\eqref{bijlcm}. 

\begin{proposition}\label{recordlcm}
 For $k,n,d,D\ge 1$,
 \begin{equation*}\sum a_{i_1,j_1,\ell_1} \cdots a_{i_n,j_n,\ell_n} = a_{k,nd,nD}, 
 \end{equation*}
where the sum on the left-hand side is over $i_1,\ldots ,i_n\ge 1$, $j_1,\ldots ,j_n\ge 1$ and $\ell_1,\ldots ,\ell_n\ge 1$  such that
\[  i_1+ \cdots +i_n=k,\qquad \gcd\{ j_1, \ldots , j_n\} =d,   \quad {\textit{and}} \quad  \mathrm{lcm} \{ \ell_1,\ldots ,\ell_n\} = D . \]
 \end{proposition}
 
\section{Asymptotic growth of coefficients}
\label{sect6}

We now turn to asymptotics, and begin with the proof of Theorem~\ref{asymp}.

\vspace{.1in}

\begin{proof}[Proof of Theorem~\ref{asymp}]
Theorem~\ref{mainresult} gives $M(A(x))=x$, and rearranging this equation, we get
\begin{equation}\label{phiform}
A(x) = x\, \phi(A(x)),
\end{equation}
where
\begin{equation}\label{phiM}
\phi(u)=\frac{u}{M(u)}=\frac{1}{G(u)},\qquad\qquad G(u)=\sum_{k\ge 1} \mu(k)u^{k-1}.
\end{equation}
The claim will follow from a result of Flajolet and Sedgwick~\cite[Theorem VI.6, pp.~404--405]{fs09}, which gives the asymptotic behaviour of the coefficients of $A(x)$ in terms of~$\phi(u) = \sum_{n\ge 0} \phi_n u^n$. Their theorem requires that $\phi$ be a nonlinear function satisfying the following conditions~($\bf{H}_{\bf{1}}$) and~($\bf{H}_{\bf{2}}$):

\vspace{.1in}

\textbf{Condition} ($\bf{H}_{\bf{1}}$). The function $\phi(u)$ is analytic at $u=0$ and satisfies
\[ \phi(0)\neq 0;\qquad\qquad \phi_n\ge 0, \text{ for } n\ge 0;\qquad\qquad \phi(u)\not\equiv \phi_0 + \phi_1 u. \]

\vspace{.1in}

\textbf{Condition} ($\bf{H}_{\bf{2}}$). Within the open disc of convergence of $\phi$ at $0$, say of radius $R>0$, there exists a (then necessarily unique) positive solution to the characteristic equation
\begin{equation}\label{chareqn}
\phi(u)-u\, \phi'(u)=0, \qquad 0<u<R.
\end{equation}

\vspace{.1in}

We now prove that $\phi$ satisfies the above conditions. Clearly $\phi$ is nonlinear, so~$\phi(u)\not\equiv \phi_0 + \phi_1 u$. Then note that all coefficients in the series $G$ defined in~\eqref{phiM} have absolute value equal to $0$ or $1$, so $G(u)$ converges for all $\vert u\vert <1$. Also, the zero of $G$ of smallest absolute value in $(-1,1)$ is given by
\[ \alpha \doteq 0.580294623807326723064776237226780436649 , \]
(an eight-digit approximation was previously given by Fr\"oberg~\cite{f66}). Hence, for~($\bf{H}_{\bf{1}}$), we obtain that $\phi(u)$ is analytic at $u=0$, and that $\phi(0)=\frac{1}{G(0)}=1\neq 0$. Also, for~($\bf{H}_{\bf{2}}$), the open disc of convergence of $\phi$ at $0$ has radius $R= \alpha$, from~\eqref{phiM}. Then in order to solve the characteristic equation~\eqref{chareqn}, we obtain from~\eqref{phiM} that
\[ \phi(u)-u\, \phi'(u) = \frac{1}{G(u)} + \frac{uG'(u)}{G^2(u)} =\frac{G(u)+uG'(u)}{G^2(u)} = \frac{M'(u)}{G^2(u)}, \]
and hence the unique positive solution to the characteristic equation~\eqref{chareqn} must be a zero of $M'$. Now there are two zeros of $M'$ in $(-\alpha,\alpha)$, given by
\begin{align*}
\tau & \doteq \;\;\; 0.32299391330283353998122564696308569320205174841752276244233373344634953499, \\
\beta & \doteq -0.562976540744649358189645954216416402249939799218087618317349878994076506622,
\end{align*}
(eight-digit approximations were previously given by Fr{\"o}berg \cite{f66}). As guaranteed in the statement of~($\bf{H}_{\bf{2}}$) above, exactly one of these roots, namely $\tau$, is positive.

Now, to prove the remaining condition of~($\bf{H}_{\bf{1}}$), that~$\phi_n\ge 0$ for~$n\ge 0$, we consider the contour integral
\[ I_n = \frac{1}{2\pi i}\int_{|z|=0.7} \frac{\phi(z)}{z^{n+1}}dz,\qquad n\ge 0. \]
From Fr\"oberg~\cite{f66}, the only zero of $G$ in the complex circle $|z|\le 0.7$ is at $z=\alpha$, so, using~\eqref{phiM},  there are two poles for $I_n$, at $0$ and $\alpha$. Also from Fr\"oberg~\cite{f66},  $\alpha$ is {\em not} a zero of $G'$, so $\alpha$ is a simple pole for $I_n$. Hence by Cauchy's theorem, we obtain
\[ I_n = \phi_n + \frac{1}{\alpha^{n+1} G'(\alpha)} = \phi_n + R_n, \]
where $\phi_n$ is the residue at $0$, and $R_n = \frac{1}{\alpha^{n} M'(\alpha)}$ is the residue at $\alpha$. But~$M'(\alpha) \doteq - 1.5863869$,~so
\[ |R_n| \doteq \frac{1}{(0.5802946)^n (1.5863869)} > \frac{1}{(0.6)^n (1.6)}. \]
Also, by the ML inequality (aka the ``estimation lemma'') we have
\begin{align*}  |I_n| &= \frac{1}{2\pi}\left| \int_{|z|=0.7} \frac{\phi(z)}{z^{n+1}}dz  \right|  \le \frac{2\pi (0.7)}{2\pi} \max_{|z|=0.7} \left|  \frac{ \phi(z)}{z^{n+1}}\right| =(0.7) \max_{|z|=0.7} \frac{1}{| z^{n}M(z)|}\\
 &= (0.7)^{1-n} \max_{|z|=0.7} \frac{1}{|M(z)|} = \frac{ (0.7)^{1-n} }{  \min_{|z|=0.7} |M(z)|  }.
\end{align*}
But from Fr\"oberg~\cite{f66}, $|M(z)|$ is minimized on the complex circle~$|z|=0.7$ at $z=0.7$, and we have~$|M(0.7)| \doteq 0.2582108$. Therefore
\[  |I_n| \le \frac{(0.7)^{1-n}}{ 0.2582108 } <  \frac{(0.7)^{1-n}}{ 0.25 } = \frac{2.8}{(0.7)^n}, \]
and combining the above inequalities for $|R_n|$ and $|I_n|$, we get
\begin{align*}
\frac{|I_n|}{|R_n|} &< (2.8)(1.6) \left( \frac{0.6}{0.7}  \right)^n = \exp \Big( \log (4.48) + n \big( \log (0.6) - \log (0.7)  \big) \Big) \\
&\doteq \exp \Big( 1.499623 - 0.154151 \, n \Big) < \exp \big( 1.5 - 0.15 \, n  \big) .
\end{align*}
It follows immediately that $|I_n|<|R_n|$ for~$n\ge 10$. But~$|I_n|=|\phi_n+R_n|$, and~$I_n$,~$R_n$,~$\phi_n$ are all real, with $R_n<0$, so we have $\phi_n \ge 0$ for~$n\ge 10$. To prove that $\phi_n\ge 0$ for $n=0,1,\ldots ,9$, we have simply used Maple, and find that
\[ \phi(u) = 1+u+2u^2+3u^3+6u^4+9u^5+17u^6+28u^7+50u^8+83u^9 + \mathcal{O}\big( u^{10}\big) . \] 

We have now proved that $\phi$ satisfies all the conditions required for Theorem VI.6 of~\cite{fs09}, which then gives
\begin{equation}\label{expnA}
A(x) = \tau - d_1 \Big( 1-\frac{x}{\rho}\Big)^{1/2} + \sum_{j \ge 2} (-1)^j d_j \Big( 1-\frac{x}{\rho}\Big)^{j/2},
\end{equation}
where $\rho = \frac{\tau}{\phi(\tau)}=M(\tau)$, $d_1= \sqrt{\frac{2\phi(\tau)}{\phi''(\tau)}}=\sqrt{-\frac{2M(\tau)}{M''(\tau)}}$, and the remaining $d_j$ are computable constants. Moreover, since $\phi$ is aperiodic (that is, $\phi(u)$ is not of the form $u^m f(u^d)$ for some series $f$ and integer $d\ge 2$), Theorem~VI.6 also gives
\[ a_k \sim \frac{d_1}{2\sqrt{\pi}}\; \rho^{-k} \, k^{-3/2} . \]
The result follows, with
\[ \gamma = \rho^{-1}\qquad\text{and}\qquad c= \sqrt{-\frac{M(\tau)}{2\pi M''(\tau)}}\; , \]
where decimal approximations to the constants are given by
\begin{align*} \rho = M(\tau) & \doteq  0.18223393401633630828235226904174072905168066104 , \\
M''(\tau) & \doteq  -4.426886252469575251674551833111186610459374194161738 , \\
\gamma & \doteq 5.48745218829746214756744529323030925532004291024 , \\
c & \doteq 0.08094229418609730035861577123355531751035381267 . 
\end{align*}
\end{proof}

Theorem~\ref{mainmthm}(b) gives a closed form for the generating series $A_m(x) = \sum_{k\ge 1} a_{k,m} x^k$, in terms of the series~$A(x) $ and~$M(x)$. Once again, we have not been able to use this result to determine a useful explicit expression for the $k$th coefficient $a_{k,m}$ in $A_m(x)$. However, in the next result, we are able to determine a precise asymptotic form for the coefficients $a_{k,m}$, following on from the proof of Theorem~\ref{asymp} above.

\begin{theorem}
For each fixed $m\ge 2$, the number $a_{k,m}$ of NECS of size $k$ with gcd $m$ is asymptotically
\[ a_{k,m} \sim m \tau^{m-1} M'(\tau^m) \, c \, \gamma^k \, k^{-3/2} , \]
where $\tau \doteq 0.3229939$, $c \doteq 0.0809423$ and $\gamma \doteq 5.4874522$.
\label{asympgcd}
\end{theorem}

\begin{proof}
From~\eqref{expnA}, taking the $m$th power, we have
\begin{equation}\label{singexpn2}
A(x)^m = \tau^m - m \tau^{m-1} d_1\, \Big(1-\frac{x}{\rho}\Big)^{1/2} + {\mathcal O} \Big(1-\frac{x}{\rho} \Big),
\end{equation}
where $\tau$, $\rho$ and $d_1$ are specified in the proof of Theorem~\ref{asymp}. In particular, since~$\tau \in(0,1)$, then~$\tau^m \in(0,1)$ for every positive integer $m$. Now using the linear expansion
\[ M(a+z) = M(a) + M'(a) z + {\mathcal O}(z^2),  \]
and~\eqref{singexpn2}, we obtain
\[ M(A(x)^m) = M(\tau^m) - m\tau^{m-1} M'(\tau^m) \, d_1 \, \Big(1-\frac{x}{\rho}\Big)^{1/2} + {\mathcal O}\Big(1-\frac{x}{\rho} \Big).  \]
But $M\big( A(x)^m \big) = A_m(x)$ from Theorem~\ref{mainmthm}(b), and we now determine the asymptotic behaviour of the coefficients in $A_m(x)$ from the above expansion. We use the technique of singularity analysis as described in \cite[Chapter VI]{fs09}. From Theorems~VI.1 on page~381 and~VI.3 on page~390, we thus obtain
\[ a_{k,m} \sim - m\tau^{m-1} M'(\tau^m) \, \frac{d_1}{\Gamma(-\frac{1}{2})} \, \rho^{-k} k^{-3/2} . \]
But $\Gamma\left( -\frac{1}{2}\right) = -2\sqrt{\pi}$, and the result follows, with
\[ \gamma = \rho^{-1} \qquad\mathrm{and}\qquad c= \frac{d_1}{2\sqrt{\pi}}= \sqrt{-\frac{M(\tau)}{2\pi M''(\tau)}} \; ,\]
with the decimal approximations of these constants as given in the proof of Theorem~\ref{asymp}.
\end{proof}

\begin{remark}
Comparing the asymptotic forms in Theorems~\ref{asymp} and~\ref{asympgcd}, we observe that
\[  a_{k,m}\sim m \tau^{m-1} M'(\tau^m) \, a_k .\]
But $\sum_{m\ge 2} a_{k,m}=a_k$ for~$k\ge 2$, so we should have
\begin{equation}\label{consistentasymp}
\sum_{m\ge 2} m \tau^{m-1} M'(\tau^m)  = 1 . 
\end{equation}
Here is a direct proof of~\eqref{consistentasymp} (which therefore provides a consistency check on our asymptotic results): A well known series identity~(see, e.g.,~\cite[p.~258]{hw60}) that follows immediately from~\eqref{mu10} is given by
\[ \sum_{k\ge 1} \mu(k) \; \frac{x^k  }{ 1-x^k  } = x.  \]
The summation on the left-hand side is referred to as the {\em Lambert series} for the M\"obius function. Rewriting the left-hand side in terms of the series $M$ itself, we obtain
\[ \sum_{m\ge 1} M(x^m) = x.  \]
Differentiating on both sides of this equation with respect to~$x$ gives
\begin{equation}\label{xident}
\sum_{m\ge 1} m x^{m-1} M'(x^m) = 1,
\end{equation}
and this holds for any~$x\in (-1,1)$. Our proof of~\eqref{consistentasymp} is then completed by substituting~$x=\tau $ in~\eqref{xident}, and noting that~$M'(\tau)=0$.
\end{remark}

\begin{remark}
Along similar lines, note that Theorem~\ref{asympgcd} also holds for~$m=1$, but the result is trivial in this case, since $M'(\tau)=0$. Hence the result for $m=1$ states that the number $a_{k,1}$ is asymptotically $0$. This is consistent with~\eqref{A1x}, which states that $A_1(x)=\sum_{k\ge 1} a_{k,m} x^k = x$, and hence $a_{k,1}=0$ for all $k\ge 2$.
\end{remark}

\section{A formula for $a_{g+n,n}$}
\label{sect7}

Inspection of the downward-sloping diagonals in Table~\ref{tabsizegcd13} suggests
that for each $n$
there is a polynomial $f_n (x)$ such that
$a_{g+n, n} = f_n (g)$ for $g >n$.   Furthermore,
it seems that $f_n(x)$ is a polynomial of degree $n$,
with leading coefficient $x^n/n!$ and constant term $0$,
and all other coefficients positive.
The first few such polynomials seem to be as follows (expressed
in the basis of binomial coefficients):
\begin{align*}
f_1 (x) &= {x \choose 1} \\
f_2 (x) &= {x \choose 2} + 3 {x \choose 1}\\
f_3 (x) &= {x \choose 3} + 6 {x \choose 2} + 10 {x \choose 1} \\
f_4 (x) &= {x \choose 4} + 9 {x \choose 3} + 29 {x \choose 2} + 39 {x \choose 1} \\
f_5 (x) &= {x \choose 5} + 12 {x \choose 4} + 57 {x \choose 3} + 138 {x \choose 2} + 160 {x \choose 1} \\
f_6 (x) &= {x \choose 6} + 15 {x \choose 5} + 94 {x \choose 4} + 324 {x \choose 3} + 654 {x \choose 2} + 691 {x \choose 1}.
\end{align*}
Furthermore it appears that the $\ell$'th differences
of the coefficients of the $\ell$'th column above is $3^\ell$.

We now explain these empirical observations.  We start with the following lemma about formal power series.

\begin{lemma}
Suppose $F(x)=1+c_1x+c_2x^2+\cdots$ is a formal
power series with constant coefficient $1$. The coefficient of $x^m$ in $F(x)^n$ equals
\[
\sum_{k=1}^m c_{m,k} \binom nk  
\]
where $c_{m,k}$ is the coefficient of $x^m$ in $(F(x)-1)^k$. We have
$c_{m,m}=c_1^m$ and $c_{m,1}=c_m$.
Furthermore, if $c_j\geq 0$ for all $j$, then $c_{m,k} >0$ for all $k$, $1\leq k\leq m$. 
\label{and}
\end{lemma}

\begin{proof} We have
\[
F(x)^n = (1+(F(x)-1))^n = \sum_{k=0}^n \binom nk (F(x)-1)^k.
\]
As $F(x)-1$ begins with an $x$-term, we see that the coefficient of $x^m$ in $(F(x)-1)^k$ is $0$ for all $k>m$.
The first part of the result follows . By the multinomial theorem,
\[
c_{m,k}  = \sum_{\substack{r_1,r_2,\ldots , r_m\in \mathbb \bbN \\ r_1+2r_2+\cdots +mr_m=m \\ r_1+\cdots+r_m=k}}
 \binom{k}{r_1,\ldots ,r_m } \prod_{j=1}^m  c_j^{r_j} .
\]
When $k=m$ then $r_1=m$ and the other $r_j=0$, so $c_{m,m}=c_1^m$.
When $k=1$ then $r_m=1$ and the other $r_j=0$, so $c_{m,1}=c_m$.
\end{proof}

Now let $f_n(g) = a_{g+n,g}$, the number of NECS of size $g+n$ and gcd $g$, so that 
\[
\sum_{n\geq 0} f_n(g)x^n=A_g(x)/x^g = \sum_{m\geq 1} \mu(m) (A(x)^m/x)^g,
\]
where we have used Eq.~\eqref{niceq}.
Now $A(x)$ equals $x$ plus higher order terms. Hence, if $m\geq 2$ and $g>n$, then the coefficient of $x^n$ in  $(A(x)^m/x)^g$ is $0$. Therefore
\begin{center}
If $g>n$ then $f_n(g)$ is the coefficient of $x^n$ in $(A(x)/x)^g$.
\end{center}
Now $A(x)/x=1+x+3x^2+\cdots$ has all positive coefficients (cf.~Remark~\ref{rmk3}).   Therefore we may apply Lemma~\ref{and} with $F(x) = A(x)/x$
to obtain the following:
\begin{corollary} Fix an integer $n\geq 1$. If $g>n$ then the number of NECS with size $g+n$ and gcd $g$ is given by a polynomial
\[
\sum_{k=1}^n c_{n,k} \binom gk  .
\]
in which the coefficients $c_{n,k} $ are all positive and, written as a polynomial, 
the leading term is $g^n/n!$ and the constant term is $0$.
We also have $c_{n,1}=a_{k+1}$.
\end{corollary}

Let us now turn to a better formula for $c_{m+\ell,m}$ for fixed $\ell\geq 1$, and with $m\geq \ell$.   To do so, we
write $A(x)/x =1+x+xB(x)$ where $x$ divides $B(x)$. Now 
$c_{m+\ell,m}$ is the coefficient of $x^{m+\ell}$ in $({{A(x)} \over {x}}-1)^{m}= (x(1+B(x)))^{m}$, which equals
the coefficient of $x^\ell$ in $(1+B(x))^{m}$. As $x$ divides $B(x)$ this implies that
\begin{center}
$c_{m+\ell,m}$ equals the coefficient of $x^\ell$ in $ \sum_{h=0}^\ell \binom{m} h   B(x)^h$.
\end{center}

The $\ell$th backward difference of $(c_{m+\ell,m})$ is
\[
\sum_{j=0}^\ell \binom \ell j (-1)^j c_{m+\ell-j,m-j}  ,
\]
 which equals the coefficient of $x^\ell$ in 
 \[
\sum_{j=0}^\ell \binom \ell j (-1)^j \sum_{h=0}^\ell \binom{m-j} h   B(x)^h =
\sum_{h=0}^\ell  \left( \sum_{j=0}^\ell \binom \ell j (-1)^j \binom{m-j} h \right)  B(x)^h .
\]

From the theory of finite differences we know, for all $m$, that
\[
\sum_{j=0}^\ell \binom \ell j (-1)^j \binom{m-j} h = 
\begin{cases}
0, &\text{ if } 0\leq h< l; \\
1, &\text{ if } h=l.
\end{cases}
\]
Therefore the $\ell$th backward difference of $(c_{m+\ell,m})$ equals the coefficient of $x^\ell$ in $B(x)^\ell$, which is the leading coefficient, and so equals $c_2^\ell$. In our special case this gives,
 if $\ell\geq 1$ and $m\geq \ell$, then
\[
\sum_{j=0}^\ell \binom \ell j (-1)^j c_{m+\ell-j,m-j} = 3^\ell,
\]
as observed in the data.

\section{Open Problems}
\label{sect8}

In this section we list three related problems for which we currently have no solution.

\begin{problem}
Suppose that, instead of counting distinct NECS of size $k$, we count equivalence classes under ``shift''.  That is, we consider two NECS to be identical if one can be transformed into the other by a transformation of the form $x = x' + C$, for some integer constant $C$.   How many equivalence classes are there?
We list the number $s(k)$ for $1 \leq n \leq 12$:
\begin{center}
\begin{tabular}{c|cccccccccccc}
$k$ & 1 & 2 & 3 & 4 & 5 & 6 & 7 & 8 & 9 & 10 & 11 & 12 \\
\hline
$s(k)$ &  1 & 1 & 2 & 4 & 10 & 26 & 75 & 226 & 718 & 2368 & 8083 & 28367
\end{tabular}
\end{center}
What is a good formula for $s(k)$?  What is the asymptotics of~$s(k)$? 
\end{problem}

\begin{problem}
Suppose we consider those NECS of size $k$, and ask how many distinct values of the lcm parameter they can take on.  Call the resulting sequence $t(k)$.  The first few values are given below.
\begin{center}
\begin{tabular}{c|cccccccccccc}
$k$ & 1 & 2 & 3 & 4 & 5 & 6 & 7 & 8 & 9 & 10 & 11 & 12\\
\hline
$t(k)$ &  1 & 1 & 2 & 3 & 6 & 8 & 15 & 18 & 31 & 35 & 56 & 62
\end{tabular}
\end{center}
What is a good formula for~$t(k)$?  What is the asymptotics of~$t(k)$?
\end{problem}

\begin{problem}\label{ECSenum}
Suppose we consider the enumeration of ECS instead of NECS. Hence let~$b_k$ denote the number of ECS of size~$k$,~$k\ge 1$, and~$b_{k,m}$ denote the number of ECS of size~$k$ and gcd~$m$,~$k,m\ge 1$. Define the generating functions
\[ B(x) = \sum_{k\ge 1} b_k x^k , \qquad\qquad B_m(x) = \sum_{k\ge 1} b_{k,m} x^k, \quad m\ge 1 . \]
It is reasonably straightforward to prove that the analogue of Theorem~\ref{mainmthm}(a) holds for these series, so we have $B(x)^n = \sum_{d \ge 1} B_{nd}(x)$, for~$n\ge 1$. Then, applying M\"obius inversion, we obtain~$M(B(x)^m)=B_m(x)$, $m\ge 1$, the analogue of Theorem~\ref{mainmthm}(b). Specializing to the case~$m=1$ gives~$M(B(x))=B_1(x)$, or equivalently
\begin{equation}\label{MB1}
B(x)=A(B_1(x)),
\end{equation}
where~$B_1(x)$ is the generating function for the ECS with gcd equal to~$1$. But it turns out that
\[ B_1(x) = x + 30 \, x^{13} + {\mathcal O}\big( x^{14} \big) , \]
a substantially different situation from the NECS, where we had~$A_1(x)=x$. In fact, the~$b_{13,1} = 30$ ECS with gcd~$1$ are the only ECS of size at most~$13$ that are not also~ECS. Thus, in Table~\ref{tabsizegcd13}, the values of~$a_{k,n}$ that appear in row~$k$ and column~$n$ are equal to~$b_{k,n}$ everywhere except for row~$13$ and column~$1$. However, for larger values of~$k$ the gap between the numbers of ECS and NECS grows rapidly.

Now, from \eqref{expnA} and~\eqref{MB1}, together with the fact that the coefficients of~$B_1(x)$ are non-negative, we can deduce that the asymptotic growth rate for~$b_k$ is strictly larger than~$\gamma\doteq 5.4874522$, the growth rate for~$a_k$ given in Theorem~\ref{asymp}. In other words, the growth rate for the number of ECS is strictly larger than the growth rate for the number of NECS. However, we have no further information about asymptotics for~$b_k$. What would we have to know about the growth rate for the coefficients of~$B_1(x)$ in order to deduce asymptotics for the numbers~$b_k$ from the functional equation~\eqref{MB1}? 
\end{problem}

\section{Comments}
\label{sect9}

This paper was originally motivated by a problem dealing with infinite periodic sequences
of constant gap.  These are maps from $\mathbb{N}$ to a finite alphabet of
size $k$, say $\Sigma_k := \{ 0,1, \ldots, k-1 \}$, with the property that 
for each $i \in \Sigma_k$ there exists a constant $c_i$ such that the occurrences of $i$ lie in an arithmetic progression of difference $c_i$.  For example,
the infinite periodic sequence $(0102)^\omega = 010201020102 \cdots $ is of constant gap with $k = 3$. These sequences have been studied,
e.g., in \cite{g73,h96,h00,agh}.

David W. Wilson \cite{w17} and JS independently conjectured, on the basis of numerical evidence, that the reversion of the M\"obius series~$M$ counts the number of ECS.  This is incorrect; as we have seen, this reversion instead counts the (strict) subclass of NECS.  The first place where these two sequences differ is at $k = 13$, where the number of NECS is $7266979$ (e.g., this is the total of the entries in row~$k=13$ of Table~\ref{tabsizegcd13}), but the number of ECS is $7267009$ (which is~$30$ larger---see the discussion in Problem~\ref{ECSenum} above).

\section*{Acknowledgments}

We thank Cam Stewart for helpful remarks.  Vincent Jost kindly informed us by e-mail on November 21 2017 that an earlier version of our paper was flawed, because we did not distinguish correctly between ECS and the special case of NECS.  We gratefully thank him for his comments.  We also thank the referee for pointing out the erratum in \cite{fs09} mentioned in Remark~\ref{schro}.

\bibliographystyle{plain}

\begin{thebibliography}{NZ7457}

\bibitem[AGH00]{agh}
E. Altman, B. Gaujal, and A. Hordijk,
{\em Balanced sequences and optimal routing},
J. ACM \textbf{47} (2000), 752--775.

\bibitem[B76a]{b76a}
N. Burshtein,
{\em On natural exactly covering systems of congruences having
moduli occurring at most $M$ times},
Discrete Math. \textbf{14} (1976), 205--214.

\bibitem[B76b]{b76b}
N. Burshtein,
{\em On natural exactly covering systems having moduli occurring
at most twice}, J. Number Theory \textbf{8} (1976), 251--259.

\bibitem[E50]{e50}
P. Erd\H{o}s,
{\em On integers of the form $2^k + p$ and related problems},
Summa. Brasil. Math. \textbf{2} (1950), 192--210.

\bibitem[E52]{e52}
P. Erd\H{o}s,
{\em Egy kongruenciarendszerekr\H{o}l sz{\'o}l{\'o} probl{\'e}m{\'a}r{\'o}l},
Mat. Lapok \textbf{3} (1952), 122--128.

\bibitem[FS09]{fs09}
P. Flajolet and R. Sedgewick,
{\em Analytic Combinatorics},
Cambridge University Press, 2009.

\bibitem[F73]{f73}
A. S. Fraenkel, {\em A characterization of exactly covering congruences},
Discrete Math. \textbf{4} (1973), 359--366.

\bibitem[F72]{f72}
J. Friedlander,
{\em On exact coverings of the integers},
Israel J. Math. \textbf{12} (1972), 299--305.

\bibitem[F66]{f66}
C.-E. Fr{\"o}berg,
{\em Numerical studies of the M{\"o}bius power series},
BIT \textbf{6} (1966), 191--211.

\bibitem[G73]{g73}
R. Graham,
{\em Covering the positive integers by disjoint sets of
the form $\{[n\alpha + \beta]\} : n = 1,2, \ldots$},
J. Combin. Theory \textbf{15} (1973), 354--358.

\bibitem[HL16]{hl16}
G. H. Hardy and J. E. Littlewood,
{\em Contributions to the theory of the Riemann zeta-function
and the theory of the distribution of primes},
Acta Arith. \textbf{41} (1916), 119--196.

\bibitem[HW60]{hw60}
G. H. Hardy and E. M. Wright,
{\em An Introduction to the Theory of Numbers},
4th edition, Oxford University Press, 1960.

\bibitem[H96]{h96}
P. Hubert,
{\em Propri{\'e}t{\'e}s combinatoires des suites
d{\'e}finies par le billard dans les triangles pavants},
Theoret. Comput. Sci. \textbf{164} (1996), 165--183.

\bibitem[H00]{h00}
P. Hubert,
{\em Suites {\'e}quilibr{\'e}es}, Theoret.
Comput. Sci. \textbf{242} (2000), 91--108.

\bibitem[K84]{k84}
I. Korec,
{\em Irreducible disjoint covering systems},
Acta Arith. \textbf{44} (1984), 389--395.
 
\bibitem[N71]{n71}
M. Newman,
{\em Roots of unity and covering sets},
Math. Ann. \textbf{191} (1971), 279--282.

\bibitem[NW78]{nw78}
A. Nijenhuis and H. S. Wilf,
{\em Combinatorial Algorithms}, 2nd edition,
Academic Press, 1978.

\bibitem[NZ74]{nz74}
B. Novak and S. Zn{\'a}m,
{\em Disjoint covering systems},
Amer. Math. Monthly \textbf{81} (1974), 42--45.

\bibitem[P74]{p74}
S. Porubsk{\'y}, {\em Natural exactly covering systems of congruences},
Czechoslovak Math. J. \textbf{24} (1974), 598--606.

\bibitem[P81]{p81}
S. Porubsk{\'y}, {\em Results and problems on covering systems of
residue classes}, Mitt. Math. Sem. Giessen {\bf 150} (1981), 1--85.

\bibitem[PS02]{ps02}
S. Porubsk{\'y} and J. Sch{\"o}nheim,
{\em Covering systems of Paul Erd\H{o}s:  past, present and future},
in Paul Erd\H{o}s and his Mathematics, Vol.~I, Bolyai Society
Mathematical Studies \textbf{11} (2002), 581--627.

\bibitem[S15]{s15}
O. Schnabel, {\em On the reducibility of exact covering systems},
INTEGERS \textbf{15} (2015), Paper \#A34.

\bibitem[S86]{s86}
R. J. Simpson,
{\em Exact coverings of the integers by arithmetic progressions},
Discrete Math. \textbf{59} (1986), 181--190. 

\bibitem[SL]{sl}
N. J. A. Sloane et al., {\em The On-Line Encyclopedia
of Integer Sequences}, available at
\url{http://oeis.org}.

\bibitem[S05]{s05}
Z. W. Sun, {\em Problems and results on covering systems},
a talk given in 2005, available at \url{http://maths.nju.edu.cn/~zwsun/Cover.pdf}.

\bibitem[W17]{w17}
David W. Wilson, remark in \cite{sl}, sequence \seqnum{A050385}, May 17 2017.

\bibitem[Z75]{z75}
S. Zn{\'a}m,
{\em A simple characterization of disjoint covering systems},
Discrete Math. \textbf{12} (1975), 89--91.

\bibitem[Z82]{z82}
S. Zn{\'a}m, {\em A survey of covering systems of congruences},
Acta Math. Univ. Comenianae \textbf{40/41} (1982), 59--79.



\end{thebibliography}

\end{document}